\documentclass[reqno]{amsart}

\usepackage{preamble}

\hypersetup{colorlinks=true,
    linkcolor=blue,
    citecolor=blue,
    pdftitle={On the packing dimension of distance sets with respect to \texorpdfstring{$C^1$}{C1} and polyhedral norms},
    pdfauthor={Iqra Altaf, Ryan Bushling, and Bobby Wilson}}

\begin{document}

\begin{abstract}
    We prove that, for every polyhedral or $C^1$ norm on $\R^d$ and every set $E \subseteq \R^d$ of packing dimension $s$, the packing dimension of the distance set of $E$ with respect to that norm is at least $\tfrac{s}{d}$. One of the main tools is a nonlinear projection theorem extending a result of M.~J\"{a}rvenp\"{a}\"{a}. An explicit construction follows, demonstrating that these distance sets bounds are sharp for a large class of polyhedral norms. 
\end{abstract}

\title[Packing dimension of distance sets for $C^{1\!}$~and polyhedral norms]{On the packing dimension of distance sets \\ with respect to $\bm{C^{1\!}}$ and polyhedral norms}
\author[Altaf]{Iqra Altaf}
\address{Department of Mathematics \\ University of Chicago \\ 5734 S University Avenue, Room 108 \\ Chicago, IL, 60637-1505}
\email{iqra@uchicago.edu}
\author[Bushling]{Ryan Bushling}
\address{Department of Mathematics \\ University of Washington, Box 354350 \\ Seattle, WA 98195-4350}
\email{reb28@uw.edu}
\author[Wilson]{Bobby Wilson}
\address{Department of Mathematics \\ University of Washington, Box 354350 \\ Seattle, WA 98195-4350}
\email{blwilson@uw.edu}
\thanks{B. W. is supported by NSF CAREER Fellowship, DMS 2142064.}

\maketitle

%%%%%%%%%%%%%%%%%%%%%%%%%%%%%%%
%%% Section 1: Introduction %%%
%%%%%%%%%%%%%%%%%%%%%%%%%%%%%%%

\section{Introduction} \label{s:intro}

\noindent For a norm $\| \cdot \|_*$ on $\R^d$ and set $E \subseteq \R^d$, let
\begin{equation*}
    \Delta^*(E) := \big\{ \| y-x \|_* \in [0,\infty) \!: x,y \in E \big\}
\end{equation*}
be the \textit{distance set} of $E$ with respect to that norm. Given a fixed $x \in \R^d$, we also let
\begin{equation*}
    \Delta_x^*(E) := \big\{ \| y-x \|_* \in [0,\infty) \!: y \in E \big\}
\end{equation*}
be the \textit{pinned distance set} of $E$ at $x$. The classical \textit{Falconer distance conjecture} \cite{falconer1985hausdorff} posits that, for $E$ (say, Borel),
\begin{equation*}
    \hdim E > \frac{d}{2} \qquad \Longrightarrow \qquad \hdim \Delta(E) = 1,
\end{equation*}
where $\hdim$ denotes Hausdorff dimension and $\Delta$ the distance set operator for the Euclidean norm. More quantitatively, one frequently seeks lower bounds on $\hdim \Delta(E)$ or $\hdim \Delta_x(E)$ in terms of $\hdim E$. For recent progress on the distance set conjecture, see \cite{du2023new,fiedler2023dimension,fiedler2024pinned}. In particular, \cite{du2023new} contains a clear and thorough history of the problem.

This paper may be seen as a follow-up to \cite{altaf2023distance}, in which the authors treat the Falconer distance problem for general norms, emphasizing in particular the case of polyhedral norms. The main result therein is the following.

\begin{nthm}[\cite{altaf2023distance} Theorems 1.1 \& 1.4]
    Let $\| \cdot \|_*$ be a norm on $\R^d$ and let $E \subseteq \R^d$. Then
    \begin{equation} \label{eq:main-hausdorff}
        \hdim \Delta_x^*(E) \geq \hdim E - (d-1)
    \end{equation}
    for all $x \in \R^d$. This bound is sharp for polyhedral norms in the sense that, if $\| \cdot \|_P$ is a polyhedral norm on $\R^d$, then for any $s \in [d-1,d]$, there exists a compact set $E \subset \R^d$ with $\hdim E = s$ such that
    \begin{equation*}
        \hdim \Delta^P(E) = s - (d-1).
    \end{equation*}
\end{nthm}

\subsection{Main results} \label{ss:results} Distance sets with respect to polyhedral norms have also been considered in \cite{konyagin2006distance,falconer2004dimensions,konyagin2004separated,bishop2021falconer}, and in addition \cite{bishop2021falconer} presents strictly convex $C^1$ norms with pathological distance sets. The works \cite{orponen2017distance,keleti2019new,fiedler2023dimension,fiedler2024pinned} address problems concerning the \textit{packing} dimension $\pdim$ of distance sets, and here we consider the packing dimension of distance sets with respect to polyhedral and $C^1$ norms. Our main theorem is the following.

\begin{thm} \label{thm:main}
    Let $\| \cdot \|_*$ be either a $C^1$ or polyhedral norm on $\R^d$ and let $E \subseteq \R^d$. For every $\eps>0$, there exists $ x \in E$ such that
    \begin{equation} \label{eq:main-packing}
        \pdim \Delta_{x}^*(E) \geq \frac{1}{d} \pdim E - \eps.
    \end{equation}
    In particular,
    \begin{equation*}
        \pdim \Delta^*(E) \geq \frac{1}{d} \pdim E.
    \end{equation*}
\end{thm}

In both the $C^1$ and the polyhedral case, the normal vector field to the surface the unit ball is regular enough to establish reasonable transversality conditions on the pinned distance functions $y \mapsto \|x-y\|_*$. Just as \eqref{eq:main-hausdorff} was sharp for every polyhedral norm in a strong way, \eqref{eq:main-packing} is also sharp for a large class of polyhedral norms.

\begin{thm} \label{thm:sharp}
    Let $\| \cdot \|_P$ be a polyhedral norm such that the unit normal vectors to the faces of the polyhedron are rationally dependent and let $s \in [0,d]$. There exists a compact set $E \subset \R^d$ with $\pdim E = s$ such that
    \begin{equation} \label{eq:sharp}
        \pdim \Delta^P(E) = \frac{s}{d}.
    \end{equation} 
\end{thm}

A main tool in the proof of Theorem \ref{thm:main} is a nonlinear projection estimate extending a result of J\"{a}rvenp\"{a}\"{a} \cite{jarvenpaa1994upper}. Let $\mathbf{p} = (\mathbf{p}_i)_{i=1}^d$ be a family of functions $\mathbf{p}_i \!: E \to \R$ defined on a set $E \subseteq \R^d$ and let $\ubdim$ denote upper box dimension.
\begin{defn}[Weak transversality]
    We say that $\mathbf{p}$ is \textit{weakly transversal} if, for every bounded set $F \subseteq E$, its restriction $\mathbf{p}|_F$ is bounded and has fibers of upper box dimension $0$. That is:
    \begin{enumerate}[label={\normalfont \textbf{\roman*.}}, topsep=-3pt, noitemsep]
        \item $\mathbf{p}(F)$ is bounded; and
        \item for every $\delta \in (0,1)$, there exists a natural number $M(\delta) \lesssim_\eps \delta^{-\eps}$ for every $\eps > 0$ such that the following holds: for every $\xi \in \R^n$, there exist $m \leq M(\delta)$ many points $x_1,\dots,x_m \in \R^n$ such that
        \begin{equation} \label{eq:fiber-separation}
            \mathbf{p}^{-1}\big( B(\xi,\delta) \big) \cap F \subseteq \bigcup_{k=1}^m B(x_k,\delta).
        \end{equation}
    \end{enumerate}
\end{defn}

The first condition simply ensures that $\ubdim \mathbf{p}(F)$ is defined whenever $\ubdim F$ is. The second condition ensures that the fibers of the functions $\mathbf{p}_i$ intersect fairly transversely at most points; it is always satisfied, for example, when $\mathbf{p}$ is locally invertible with Lipschitz inverses. Either $\delta$ (or both) in \eqref{eq:fiber-separation} can be replaced with a number $\sim_\eps \delta$.

With this terminology, we have the following result.

\begin{prop} \label{prop:nonlinear-jarvenpaa}
    Let $\mathbf{p} = \big( \mathbf{p}_i \!: E \to \R \big)_{i=1}^d$ be a weakly transversal family, let $n \in \{1,\dots,d\}$, and for every string of indices $i_1,\dots,i_n$ let $\mathbf{p}_{i_1,\dots,i_n} := \big( \mathbf{p}_{i_j} \!: E \to \R \big)_{j=1}^n$. Then for every set $F \subseteq E$,
    \begin{equation} \label{eq:box-jarvenpaa}
        \frac{n}{d} \ubdim F \leq \max_{1 \leq i_1 < \cdots < i_n \leq d} \ubdim \mathbf{p}_{i_1,\dots,i_n}(F)
    \end{equation}
    provided $F$ is bounded, and
    \begin{equation} \label{eq:packing-jarvenpaa}
        \frac{n}{d} \pdim F \leq \max_{1 \leq i_1 < \cdots < i_n \leq d} \pdim \mathbf{p}_{i_1,\dots,i_n}(F).
    \end{equation}
\end{prop}

We only require the result for $n=1$, and it is possible to prove this directly by a simple covering argument. However, we have included the full statement in the interest of generality and elected a proof that leverages J\"{a}rvenp\"{a}\"{a}'s result on orthogonal projections.

If $\| \cdot \|_*$ is a strictly convex norm on $\R^2$, the ``projection maps" $\mathbf{p}_i(x) := \| x - x_i \|_*$ are weakly transversal for any two distinct $x_1,x_2 \in \R^2$, whereby the distance set bound of Theorem \ref{thm:main} follows immediately from Proposition \ref{prop:nonlinear-jarvenpaa}. In higher dimensions, one can establish a similar result for strictly convex norms for $d$ pins $x_1,\dots,x_d$ in general position. Things are not so simple for norms with more degenerate curvature. Establishing transversality for \textit{some} set of pins on a large portion of $E$ becomes the main hurdle in proving distance sets estimates for such norms.

\subsection{Outline} \label{ss:outline} After \S \ref{s:prelim} lays out the necessary background on the upper box and packing dimensions, \S \ref{s:nonlinear-jarvenpaa} presents the proof of Proposition \ref{prop:nonlinear-jarvenpaa} and \S \ref{s:polyhedral-proof} gives the proof of Theorem \ref{thm:main} for polyhedral norms. Next, the transversality conditions necessary to apply Proposition \ref{prop:nonlinear-jarvenpaa} for projections defined by $C^1$ norms are established in \S \ref{s:C1-transversality}, and in \S \ref{s:C1-proof} the proof of Theorem \ref{thm:main} for $C^1$ norms is completed. Finally, the sharp example of Theorem \ref{thm:sharp} is constructed in \S \ref{s:sharp}.

%%%%%%%%%%%%%%%%%%%%%%%%%%%%%%%%
%%% Section 2: Preliminaries %%%
%%%%%%%%%%%%%%%%%%%%%%%%%%%%%%%%

\section{Preliminaries} \label{s:prelim}

\noindent Let $E \subset \R^d$ be any bounded set and $N(E,\delta)$ its \textit{covering number} at scale $\delta$\textemdash the minimal number of $\delta$-balls required to cover $E$. The \textit{upper box dimension} of a bounded set $E \subset \R^d$ is defined by
\begin{equation*}
    \ubdim E := \limsup_{\delta \to 0} \frac{\log N(E,\delta)}{-\log \delta}.
\end{equation*}
Note that upper box dimension is finitely stable not but countably stable: that is, if $E = \bigcup_{i=1}^m E_i$, then $\ubdim E = \max_{1 \leq i \leq m} \ubdim E_i$ provided $m < \infty$, whereas it can happen that $\ubdim E > \sup_{1 \leq i < \infty} \ubdim E_i$. Packing dimension can be defined simply by ``forcing" upper box dimension to be countably stable: given any set $E \subseteq \R^d$ (not necessarily bounded), we let
\begin{equation} \label{eq:box-to-packing}
    \pdim E := \inf \left\{ \sup_{1 \leq i < \infty} \ubdim E_i \!: E \subseteq \bigcup_{i=1}^\infty E_i \right\},
\end{equation}
where the infimum is taken over all countable covers of $E$ by bounded sets. It follows immediately that $\ubdim E \geq \pdim E$ for every bounded set $E$, and it is also not difficult to show that $\pdim E \geq \hdim E$.

We sometimes use the standard notation $A \lesssim_\alpha B$ to indicate that $A \leq c_\alpha \+ B$ for some constant $c_\alpha$ depending only on the parameter (or collection of parameters) $\alpha$; these may be suppressed from the notation where the parameters are not especially important. The notation $A \sim_\alpha B$ means that both $A \lesssim_\alpha B$ and $B \lesssim_\alpha A$. Importantly, if for some $E_1, E_2 \subset \R^n$ the relation $N(E_1,\delta) \lesssim_\alpha N(E_2,\delta)$ holds for all $\delta \in (0,1)$, where $\alpha$ is independent of $\delta$, then $\ubdim E_1 \leq \ubdim E_2$.

Finally, we remark here that \S \ref{s:C1-transversality} makes heavy use of tools from convex analysis, but these will be introduced as needed. For more thorough references, see \cite{cioranescu2012geometry,clarke1990optimization}.

%%%%%%%%%%%%%%%%%%%%%%%%%%%%%%%%%%%%%%%%%%%
%%% Section 3: Proof of Proposition 1.3 %%%
%%%%%%%%%%%%%%%%%%%%%%%%%%%%%%%%%%%%%%%%%%%

\section{Proof of Proposition \ref{prop:nonlinear-jarvenpaa}} \label{s:nonlinear-jarvenpaa}

\noindent First we recall an orthogonal projection theorem of J\"{a}rvenp\"{a}\"{a} \cite{jarvenpaa1994upper}. Given linearly independent vectors $e_1,\dots,e_n \in \R^d$, let $V = V(e_1,\dots,e_n) := \Span \, \{e_1,\dots,e_n\}$ and let $P_V \!: \R^d \to \R^n$ be the orthogonal projection of $\R^d$ onto $V$.

\begin{thm}[J\"{a}rvenp\"{a}\"{a} \cite{jarvenpaa1994upper} Theorem 3.4 \& Corollary 3.5] \label{thm:orthogonal-jarvenpaa}
    If $(e_1,\ldots,e_d)$ is a basis for $\R^d$ and $1 \leq n \leq d$, then for all $E \subseteq \R^d$ there exist $1 \leq i_1 < \dots < i_n \leq d$ such that
    \begin{equation*}
        \frac{n}{d} \ubdim E \leq \ubdim P_{V(e_{i_1},\dots,e_{i_n})}(E)
    \end{equation*}
    provided $E$ is bounded, and
    \begin{equation*}
        \frac{n}{d} \pdim E \leq \pdim P_{V(e_{i_1},\dots,e_{i_n})}(E).
    \end{equation*}
\end{thm}

The statement in J\"{a}rvenp\"{a}\"{a} \cite{jarvenpaa1994upper} is for $1 \leq n \leq d-1$, but it also holds trivially for $n=d$ as stated above.

\begin{proof}[Proof of Proposition \ref{prop:nonlinear-jarvenpaa}] We first consider the case \eqref{eq:box-jarvenpaa} of upper box dimension. Let $F \subseteq E$ be a bounded set and $1 \leq n \leq d$. By Theorem \ref{thm:orthogonal-jarvenpaa}, there exist indices $1 \leq i_1 < \cdots < i_n \leq d$ such that
\begin{equation*}
    \frac{n}{d} \ubdim \mathbf{p}(F) \leq \ubdim P_{V(e_{i_1},\dots,e_{i_n})}(\mathbf{p}(F)),
\end{equation*}
where $(e_1,\dots,e_d)$ denotes the standard basis of $\R^d$. Observe that $P_{V(e_{i_1},\dots,e_{i_n})}(\mathbf{p}(F)) = \mathbf{p}_{i_1,\dots,i_n}(F)$, so it suffices to show that $\ubdim \mathbf{p}(F) \geq \ubdim F$. Let $\eps > 0$ and $\delta \in (0,1)$, and let $\{ B_i \}_{i=1}^N$ be a cover of $\mathbf{p}(F)$ by $N := N(\mathbf{p}(F),\delta)$-many $\delta$-balls. By the definition of weak transversality, for each $i$ there exists a collection $\{ B_{i,j} \}_{j=1}^{N_i}$ of $\delta$-balls such that
\begin{equation*}
    \mathbf{p}^{-1}(B_i) \cap F \subseteq \bigcup_{j=1}^{N_i} B_{i,j},
\end{equation*}
where $N_i \lesssim \delta^{-\eps}$ for every $\eps > 0$. Hence, $\{ B_{i,j} \!: 1 \leq i \leq N, \, 1 \leq j \leq N_i \}$ is a cover of $F$ by
\begin{equation*}
    \sum_{i=1}^N N_i \lesssim N\delta^{-\eps}
\end{equation*}
balls of radius $\delta$. It follows that
\begin{align*}
    \ubdim F &= \limsup_{\delta \to 0} \frac{\log N(F,\delta)}{-\log \delta} \leq \limsup_{\delta \to 0} \frac{\log(N\delta^{-\eps})}{-\log \delta} \\
    &= \limsup_{\delta \to 0} \frac{\log N(\mathbf{p}(F),\delta)}{-\log \delta} + \lim_{\delta \to 0} \frac{\log \delta^{-\eps}}{-\log \delta} \\
    &= \ubdim \mathbf{p}(F) + \eps.
\end{align*}
Since $\eps > 0$ was arbitrary, the result follows and \eqref{eq:box-jarvenpaa} holds.

We now prove the packing dimension case \eqref{eq:packing-jarvenpaa}. Let $\eps > 0$ and, for each string of indices $1 \leq i_1 < \cdots < i_n \leq d$, let $\big\{ F_{i_1,\dots,i_n}^{(j)} \big\}_{j=1}^\infty$ be a cover of $\mathbf{p}_{i_1,\dots,i_n}(F)$ by bounded sets such that
\begin{equation*}
    \sup_{1 \leq j < \infty} \ubdim F_{i_1,\dots,i_n}^{(j)} < \pdim \mathbf{p}_{i_1,\dots,i_n}(F) + \frac{\eps}{2}.
\end{equation*}
The collection of all sets of the form
\begin{equation} \label{eq:F-cover}
    \bigcap_{1 \leq i_1 < \cdots < i_n \leq d} \mathbf{p}_{i_1,\dots,i_n}^{-1}\!\Big( F_{i_1,\dots,i_n}^{(j_{i_1,\dots,i_n})} \Big), \quad 1 \leq j_{i_1,\dots,i_n} < \infty
\end{equation}
is a cover of $F$, but not necessarily by bounded sets. Therefore, we refine this to a new cover $\big\{ G^{(k)} \big\}_{k=1}^\infty$ consisting of all sets formed by intersecting the elements \eqref{eq:F-cover} with unit cubes $[n_1,n_1+1) \times \cdots \times [n_d,n_d+1)$, $n_j \in \Z$.

Taking an index $k$ such that $\ubdim G^{(k)} > \pdim F - \tfrac{d}{2n} \eps$, we have by \eqref{eq:box-jarvenpaa} a string $1 \leq i_1' < \cdots < i_n' \leq d$ such that
\begin{equation} 
\begin{aligned}
    \frac{n}{d} \pdim F - \frac{\eps}{2} &< \frac{n}{d} \ubdim G^{(k)} \leq \ubdim \mathbf{p}_{i_1',\dots,i_n'}\big( G^{(k)} \big)
\end{aligned}
\end{equation}

We note that for every $k$, 
\begin{align*}
    G^{(k)}  \subset \bigcap_{1 \leq i_1 < \cdots < i_n \leq d} \mathbf{p}_{i_1,\dots,i_n}^{-1}\!\Big( F_{i_1,\dots,i_n}^{(j_{i_1,\dots,i_n})} \Big)
\end{align*} for some collection of $j_{i_1,\dots,i_n}$. Thus for our given string, $G^{(k)} \subset \mathbf{p}_{i_1',\dots,i_n'}^{-1}(F_{i_1',\dots,i_n'}^{(j)})$ for some $j$. Therefore, $\mathbf{p}_{i_1',\dots,i_n'}\big( G^{(k)} \big) \subseteq F_{i_1,\dots,i_n}^{(j)}$, which implies
\begin{equation} \label{eq:packing-inequality}
\begin{aligned}
    \frac{n}{d} \pdim F - \frac{\eps}{2} &< \ubdim \mathbf{p}_{i_1',\dots,i_n'}\big( G^{(k)} \big) \\
    &\leq \sup_{1 \leq j < \infty} \ubdim F_{i_1',\dots,i_n'}^{(j)} < \pdim \mathbf{p}_{i_1',\dots,i_n'}(F) + \frac{\eps}{2},
\end{aligned}
\end{equation}
This string $i_1',\dots,i_n'$ depends on $\eps$, but since there are only $\binom{d}{n}$ such strings, there is at least one string for which \eqref{eq:packing-inequality} holds for arbitrarily small $\eps > 0$; hence, this concludes the proof.
\end{proof}

%%%%%%%%%%%%%%%%%%%%%%%%%%%%%%%%%%%%%%%%%%%
%%% Section 4: Proof of Polyhedral Case %%%
%%%%%%%%%%%%%%%%%%%%%%%%%%%%%%%%%%%%%%%%%%%

\section{Proof of Theorem \ref{thm:main} -- The polyhedral case} \label{s:polyhedral-proof}

\noindent Let $\|\cdot\|_*$ be a norm defined by a collection $\{v_i\}_{i=1}^M$ of vectors so that 
\begin{align*}
    \|x\|_* := \max \, \{ |\langle x, v_i \rangle| \}_{i=1}^M.
\end{align*}
We note that we must have $M \geq d$. For each $x \in \R^d$ and $v \in \R^d$, we define the cone
\begin{align*}
    C(x,v) := \left \{ y \in \R^d ~:~ \|x-y\|_* = |\langle x-y, v \rangle| \right\}.
\end{align*}
Obviously, for each $x \in \R^d$,
\begin{equation*}
    \bigcup_{j=1}^M C(x, v_i) = \R^d.
\end{equation*}
In addition, let $\{z_j\}_{j=1}^{\infty} \subset \R^d$ be a countable set, so that
\begin{align*}
    \bigcap_{j=1}^{\infty} \bigcup_{i=1}^M C(z_j, v_i) = \R^d= \bigcup_{\substack{\{z_{j_1},\dots, z_{j_M}\} \\ \{v_{i_1}, \dots, v_{i_M}\}}} \bigcap_{\ell=1}^{M}  C(z_{j_{\ell}},  v_{i_{\ell}})
\end{align*}
where in the union we allow $j_s = j_t$ and $i_s = i_t$ for $s \neq t$. For $\alpha>0$ and a countable set $G \subset \R^d$, define
\begin{align*}
    k(E, G, \alpha) &:= \max_{\substack{\{z_{j_1}, \dots, z_{j_M}\} \subset G \\ \{v_{i_1}, \dots, v_{i_M}\}}} \bigg\{ \dim(\Span(\{v_{i_1},\dots, v_{i_M}\}))~:~ \\ 
    &\hspace*{4cm} \pdim\!\Big( \textstyle E \cap \bigcap_{\ell=1}^{M}  C(z_{j_{\ell}},  v_{i_{\ell}}) \Big) \geq \pdim E -\alpha \bigg\}.
\end{align*}
This is the largest number of cones based at points of $G$ that ``see" a large (relative to $\alpha$) portion of $E$ through independent faces of the norm polyhedra $\{ y \in \R^d \!: \| y-z_j\|_* = 1 \}$.

\begin{lem}\label{lem contra1}
    Let $E \subseteq \R^d$, and $\alpha>0$, and let $G \subset \R^d$ be countable. Then there exist $k\leq k(E,G, \alpha)$ and $F \subseteq E$ with $\pdim E = \pdim F$ such that the following hold:
    \begin{enumerate}[label={\normalfont\textbf{\arabic*.}},topsep=0pt,itemsep=0pt]
        \item If $F \cap \bigcap_{\ell=1}^{M}  C(z_{\ell}, v_{i_{\ell}}) \neq \varnothing$ for any $\{z_1, \dots, z_M\} \subset G$, then 
    \begin{align}\label{eq maximal}
        \dim(\Span\{v_{i_1},\dots, v_{i_M}\})\leq k.
    \end{align}
    \item There exist $\{z_1, z_2, \dots, z_k\} \subset G$ and $\{v_{i_1}, v_{i_2}, \dots, v_{i_k}\} \subseteq \{v_1, \dots, v_M\}$ such that 
    \begin{align*}
        \dim(\Span\{v_{i_1},\dots, v_{i_k}\}) = k
    \end{align*}
    and
    \begin{align*}
    \pdim\left( F \cap \bigcap_{\ell=1}^{k}  C(z_{j_{\ell}}, v_{i_{\ell}}) \right) \geq \pdim E - \alpha.
    \end{align*}
    \end{enumerate}
\end{lem}

\begin{proof}
    Let $s = \pdim E $. The sequence $k(E, G, 2^{-m}\alpha)$ for $m \in \mathbb{Z}_+$ is a monotonically decreasing sequence of integers between 1 and $d$, so pigeonhole principle implies that there exists $m \geq 1$ such that
    \begin{equation*}
        k(E, G, \alpha) \geq k\big(E, G, 2^{-(m+1)}\alpha\big) = k(E, G, 2^{-m}\alpha) =: k.
    \end{equation*}
    For each  $\bar{z} \in G^M = G \times G \times \cdots \times G$ and $\bar{v} \in \{v_1,\dots, v_M\}^M$, define
    \begin{align*}
        E_{ \bar{z}, \bar{v}} := E \cap \bigcap_{\ell=1}^{M}  C(z_{j_{\ell}}, v_{i_{\ell}}).
    \end{align*}
    If $\dim(\Span\{v_{i_1},\dots, v_{i_M}\})> k$, then 
    \begin{align*}
        \pdim(E_{\bar{z}, \bar{v}}) \leq s-2^{-m}\alpha
    \end{align*}
    for all $\bar{z} \in G^k$. If we let $E_{\bar{v}} := \bigcup_{\bar{z} \in G^M} E_{\bar{z}, \bar{v}}$, then $\pdim E_{\bar{v}} \leq s-2^{-m}\alpha$. Therefore, if $V_{k} := \{ \bar{v} \in \{v_1,\dots, v_M\}^M~:~ \dim(\Span\{v_{i_1}, \dots, v_{i_M}\})>k \}$ then 
    \begin{align*}
         F:= E \setminus \bigcup_{\bar{v} \in V_k} E_{\bar{v}}
    \end{align*}
    has $\pdim F=\pdim E$ and the desired condition holds.

    For the second part of the statement, if $k\big(E, G, 2^{-(m+1)}\alpha\big) = k$, then by definition there exist  $\{z_1, z_2, \dots, z_k\} \subset G$ and $\{v_{i_1}, v_{i_2}, \dots, v_{i_k}\} \subseteq \{v_1, \dots, v_M\}$ such that $\dim(\Span\{v_{i_1},$ $\dots,v_{i_k}\})=k$ and 
    \begin{align*}
        \pdim\left( E \cap \bigcap_{\ell=1}^{M}  C(z_{j_{\ell}},  v_{i_{\ell}}) \right)\geq \pdim E -2^{-(m+1)}\alpha.
    \end{align*}
    Finally, $\pdim E_{\bar{v}} \leq s-2^{-m}\alpha$ implies 
    \begin{equation*}
        \pdim\left( F \cap \bigcap_{\ell=1}^{M}  C(z_{j_{\ell}},  v_{i_{\ell}}) \right)\geq \pdim E -2^{-(m+1)}\alpha\geq \pdim E -\alpha. \qedhere
    \end{equation*}
\end{proof}

\begin{prop}\label{prop penultprop}
Let $E \subset  \R^d$, $\alpha>0$, and $G \subset E$ be a countable, dense subset of $E$. Then there exists  $k \in \{1,\dots, d\}$, $\{z_1,\dots, z_k\} \subset G$, $\{w_1, w_2, \dots, w_k\} \subset \{v_1, \dots, v_M\}$, and $F \subset E$, such that $\pdim F \geq\pdim E -\alpha$, and 
    \begin{align*}
        F \subset \bigcap_{\ell=1}^k C(z_{\ell}, w_{\ell}).
    \end{align*}
Moreover, $V:=\Span\{w_1,\dots, w_k\}\in G(d, k)$ and for each $x, y \in F$
    \begin{align*}
        \|x-y\|_*= \|P_V(x-y)\|_*.
    \end{align*}
\end{prop}

\begin{proof}
Lemma \ref{lem contra1} implies that  there exist an integer $k \in \{1, \dots, d\}$,  $\{z_1, z_2, \dots, z_k\} \subset G,$ $\{w_1, w_2, \dots, w_k\} \subset \{v_1, \dots, v_M\}$, and $F_0 \subset E$ such that $\pdim F_0 = \pdim E $, 
\begin{align*}
    \dim\big(\Span\{w_1, w_2, \dots, w_k\}\big) = k
\end{align*}
and 
\begin{align*}
    \pdim\!\left( F_0 \cap \bigcap_{\ell=1}^{k} C(z_{j_{\ell}}, w_{\ell}) \right) \geq \pdim E - \alpha.
\end{align*}

Moreover, inequality \eqref{eq maximal} implies that if $F_0 \cap \bigcap_{\ell=1}^{M}  C(z_{\ell}, w_{\ell}) \neq \varnothing$ for any $\{z_1, \dots, z_M\} \subset G$ and $(v_{j_1}, \dots, v_{j_M}) \in \{v_1, \dots, v_M\}^M$, then 
\begin{align*}
    \dim\big(\Span\{v_{j_1},\dots, v_{j_M}\}\big) \leq k.
\end{align*}

Let $F:= F_0 \cap \bigcap_{\ell=1}^{k}  C(z_{j_{\ell}},  w_{\ell})$ and  $V:=\Span\{w_1,\dots, w_k\}$. If $x \in F$, $z \in G$, and $v \in \{v_1, \dots, v_M\}$ satisfy $x \in C(z, v)$ then 
\begin{align*}
    \dim\big(\Span\{w_1,\dots, w_k, v\}\big) \leq k,
\end{align*}
which implies $v \in \Span\{w_1,\dots, w_k\}$. This further implies that for any $x \in F$ and $z \in G$, there is a $v \in \{v_1, \dots, v_M\}$ such that
\begin{align*}
    \|P_V(z-x)\|_* \leq\|z-x\|_*&= |\langle z-x, v\rangle| \\
    &= |\langle P_V(z-x), v\rangle +\langle P_V^{\perp}(z-x), v\rangle| =|\langle P_V(z-x), v\rangle|\\
    &\leq \|P_V(z-x)\|_*.
\end{align*}
Since $G$ is dense in $E$, $G$ is dense in  $F$ and $\|\cdot\|$ and $P_V$ are continuous, we can conclude that for every $x, y \in F$,   
\begin{equation*}
    \|P_V(y-x)\|_*=\|y-x\|_*. \qedhere
\end{equation*}
% The set of combinations of $\{z_1, \dots, z_M\} \subset G$ and $\{v_{i_1},\dots, v_{i_M}\}$ is countable and to each $\{v_{i_1},\dots, v_{i_M}\}$ we can associate one element, $V=V(\{v_{i_1},\dots, v_{i_M}\})$, in $G(d, k)$ such that 
%     \begin{align*}
%         \Span\{v_{i_1},\dots, v_{i_M}\} \subset V.
%     \end{align*}
% Let 
%     \begin{align*}
%         \mathcal{G}:= \left\{V(\{v_{i_1},\dots, v_{i_M}\})~:~ \{v_{i_1},\dots, v_{i_M}\} \right\}
%     \end{align*}
\end{proof}

We are now ready to prove Theorem \ref{thm:main} in the case of polyhedral norms.

\begin{proof}[Proof of Theorem \ref{thm:main} in the polyhedral case]
For $E \subset \R^d$, $\alpha>0$, and $G\subset E$ with $\pdim E=s$ and with $G$ countable and dense in $E$, let $F \subset E$, $k$, $\{z_1, \dots, z_k\} \subset G$ and $\{w_1, \dots, w_k\}\subset \{v_1, \dots, v_M\}$ be the sets given by Proposition \ref{prop penultprop}. Let $V= \Span\{w_1, \dots, w_k\}$. Then Proposition \ref{prop penultprop} implies $\|x-y\|_*= \|P_V(x-y)\|_*$ for all $x, y \in F$. Therefore, the projection onto $V$ restricted to the set $F$ is a bi-Lipschitz map. Thus, $\pdim F= \pdim(P_VF) \geq s-\alpha$.

For each $j \in \{1, \dots, k\}$, let $P_j(x):= \langle w_j, x\rangle \tfrac{w_j}{\|w_j\|}$ be the orthogonal projection onto the line spanned by $w_j$. Then since $F \subset C(z_j, w_j)$, we have $P_j(x-z_j)= \|x-z_j\|_* \tfrac{w_j}{\|w_j\|}$ whenever $x \in F$. Therefore,
\begin{align*}
    \pdim P_j(F- z_j) =\pdim \Delta^*_{z_j}(F).
\end{align*}
Now Theorem \ref{thm:orthogonal-jarvenpaa} implies that for some $j \in \{1, ..., k\}$ 
\begin{align*}
    \pdim P_{j}(P_VF)\geq \frac{1}{k}  \pdim P_VF \geq \frac{s-\alpha}{k}.
\end{align*}
Since $P_VP_j= P_jP_V$ and packing dimensions are stable under translations, $\pdim P_{j}(P_VF)=\pdim P_V(P_jF)=\pdim P_V(P_j(F-z_j))=\pdim P_j(F-z_j)=\pdim \Delta^*_{z_j}(F)$. In conclusion,
\begin{align*}
    \Delta^*_{z_j}(F)\geq \frac{s-\alpha}{k}\geq \frac{s-\alpha}{d}
\end{align*}
for arbitrary $\alpha>0$.
\end{proof}

%%%%%%%%%%%%%%%%%%%%%%%%%%%%%%%%%%%%%%%%%%%%%%
%%% Section 5: Transversality for C^1 Case %%%
%%%%%%%%%%%%%%%%%%%%%%%%%%%%%%%%%%%%%%%%%%%%%%

\section{Establishing weak transversality in the \texorpdfstring{$C^1$}{C1} case} \label{s:C1-transversality}

\noindent Let $(\R^n, \|\cdot\|_*)$ be represented by $X$ as a Banach space with dual space $X^*$. For $x \in X$ and $f \in X^*$ we have the pairing $f(x)= \langle f, x\rangle$. Throughout this section, $B(x,r) := \{ y \in \R^d \!: \|x-y\|_* < r \}$ denotes a ball with respect to the norm $\| \cdot \|_*$ and $N_r(E):= \{ y \in \mathbb{R}^d\!: \|x-y\|_*< r \mbox{ for some } x \in E\}$ denotes that $r$-neighborhood of a set $E$.

We first define the subdifferential:

\begin{defn}[Subdifferential]
    For any $x \in \Omega \subset X$ and any $u: \Omega \rightarrow \R$, the \textit{subdifferential} $\partial u(x)$ of $u$ at $x$ is defined by
    \begin{align*}
        \partial u(x) := \left\{p \in X^*: \liminf _{y \rightarrow x} \frac{u(y)-u(x)-\langle p, y-x\rangle}{\|y-x\|_*} \geq 0\right\}.
    \end{align*}
\end{defn}
Of course, for a convex function, the $\liminf_{y \to x}$ can be removed without affecting the definition.  The following appears in  the text by Cioranescu \cite{cioranescu2012geometry}.

\begin{defn}[Duality mapping]
    We define the \textit{duality mapping} as a function from $X$ to the power set of $X^*$, $J: X \to 2^{X^*}$, by
    \begin{align*}
        Jx := \big\{ f\in X^*~:~ \left\langle f, x \right\rangle = \|f\|_{X^*} \|x\|_X,\, \|f\|_{X^*}= \|x\|_X \big\} 
    \end{align*}
\end{defn}

\begin{thm}[Asplund \cite{cioranescu2012geometry} Theorem I.4.4] \label{thm Asplund}
    For each $x \in X$, 
    \begin{align*}
        Jx = \partial \left( \tfrac{1}{2} \|x\|_X^2 \right)= \|x\|_X \cdot \partial \|x\|_X.
    \end{align*}
\end{thm}

For a finite collection of vectors $\{v_1, \dots, v_m\} \subset \R^d$ with $m \leq d$ we denote the $\mathcal{H}^m$-measure of the parallelepiped defined by $\{v_1, \dots, v_m\}$ by
\begin{align*}
    \left|\big( v_1~|~v_2~|~ \cdots ~|~ v_m\big) \right|.
\end{align*}
We also introduce the cone notation with respect to orthogonal projection $P_V$ for $V \in G(d,k)$, the Grassmannian of $k$-dimensional subspaces of $\R^d$. With $a \in \R^d$, $s \in (0,1)$, and $V \in G(d,k)$ we denote
\begin{align*}
    X(a,V, s) := \left\{x \in \R^d~:~ \|P_{V^{\perp}}(x-a)\|<s\|x-a\| \right\},
\end{align*}
where we recall that $\|\cdot\|$ represents the Euclidean norm.

\begin{lem} \label{lem direction}
    Let $z \in \R^d$. There exists $\eta=\eta(d, \|\cdot\|_*)$  such that 
    \begin{align*}
        z \in X(0, \Span( \nabla \|z\|_*), \eta ).
    \end{align*}
\end{lem}

\begin{proof}
    By Asplund's Theorem \ref{thm Asplund}, $\langle z, \nabla \|z\|_*\rangle = \|z\|_* $. If $\langle \cdot, \cdot \rangle$ is the Euclidean inner product and $\|\cdot\|$ is the Euclidean norm, then since all norms are equivalent in finite-dimensional normed spaces there is a constant $\Lambda=\Lambda(d, \|\cdot\|_*)\in [0, 1/2]$ such that
        \begin{align*}
            \Lambda \|z\| \leq \langle z, \nabla \|z\|_*\rangle \leq \Lambda^{-1} \|z\|
        \end{align*}
    and thus we can let $\eta= \sqrt{1-\Lambda^2}$.
\end{proof}

\begin{lem} \label{lem gradest}
    Let $\eps >0$. There exists a function, $h: (0, \infty) \to (0, \infty)$ such that $\lim_{\varepsilon \to 0} \frac{h(\varepsilon)}{\varepsilon}$ $= 0$ and, for all $x \in \R^d$ satisfying $\|x-y\|_*< \eps$,
    \begin{align*}
        \big| \|x\|_*-\|y\|_* \big|+h(\varepsilon) \geq \big|\langle \nabla \|x\|_*, x-y \rangle\big|.
    \end{align*}
\end{lem}

\begin{proof}
    Follows from the definition of gradient.
\end{proof}

Now we define transversality with respect to this norm. For $x, y \in \R^d$ and $m \in \Z_+$, define
\begin{align*}
    p_y(x)&:= \|x-y\|_*.
\end{align*}

\begin{defn}[$\bm{k}$-transversality]
    Let $k \in \{1, \dots, d\}$. We say a set $E \subseteq \R^d$ is \textit{$k$-transversal} with respect to $G \subseteq \R^d$ if there exist $L \in (0,1)$, $\{z_1,\dots,z_k\} \subseteq G$, 
    \begin{align*}
        \left|\big(\nabla p_{z_1}(x)~|~\nabla p_{z_2}(x)~|~ \cdots ~|~ \nabla p_{z_k}(x)\big)\right|\geq L
    \end{align*}
    for all $x \in E$.
\end{defn}

\begin{defn}[$\bm{(\alpha, k)}$-transversality]
    Let $k \in \{1, \dots, d\}$ and $\alpha>0$. We say a set $E \subseteq \R^d$ is \textit{$(\alpha, k)$-transversal} with respect to $G \subseteq \R^d$ if there exists  $F \subseteq E$  such that $\pdim F \geq \pdim E-\alpha$ and $F$ is $k$-transversal with respect to $G$.
\end{defn}

\begin{lem}\label{lem contra}
    If $E \subseteq \R^d$ is not $(\alpha, k)$-transversal with respect to a countable set $G \subset \R^d$, then there exists $F \subseteq E$ (with $\pdim E=\pdim F$) such that  
    \begin{align*}
        \left|\big(\nabla p_{z_1}(x)~|~\nabla p_{z_2}(x)~|~ \cdots ~|~ \nabla p_{z_k}(x)\big)\right| =0
    \end{align*}
    for all $\{z_1, \dots, z_k\} \subset G$ and $x \in F$.
\end{lem}

\begin{proof}
    Let $s = \pdim E$ and suppose $E$ is not $(\alpha, k)$-transversal.  Let $\eps_j:=2^{-j}$. For each $j$, and $\bar{z} \in G^k= G 
    \times G \times \cdots \times G$ define
        \begin{align*}
            E_{j, \bar{z}} := \left\{x \in E~:~ \left|\big(\nabla p_{z_1}(x)~|~\nabla p_{z_2}(x)~|~ \cdots ~|~ \nabla p_{z_k}(x)\big)\right| \geq 2^{-j}  \mbox{ for } \bar{z}= (z_1, \dots, z_k)\right\}.
        \end{align*}
    Since $E$ is not $(\alpha, k)$-transversal, 
        \begin{align*}
            \pdim(E_{j, \bar{z}}) \leq s-\alpha
        \end{align*}
    for all $\bar{z} \in G^k$. If we let $E_j:= \bigcup_{\bar{z} \in G^k} E_{j, \bar{z}}$, then $\pdim(E_j) \leq s-\alpha$. Therefore, if $F_j:= E \setminus E_j$, then for every $x \in F_j$ and $\{z_1, \dots, z_k\} \subset G$ we have
        \begin{align*}
             \left|\big(\nabla p_{z_1}(x)~|~\nabla p_{z_2}(x)~|~ \cdots ~|~ \nabla p_{z_k}(x)\big)\right| \leq 2^{-j}.
        \end{align*}
    Since $\cup_j E_j$ has dimension less than $s-\alpha$, $F:= \cap_j F_j$ has dimension $s$ and the proof is complete. 
\end{proof}

% \begin{lem} \label{lem induct}
%     Let $\alpha>0$ and $k \in \{1, \dots, d-1\}$. If $E \subset \R^d$ is $(\alpha, k+1)$-transversal with respect to a countable set $G$, then $E$ is $(\alpha, k)$-transversal with respect to $G$.
% \end{lem}

% \begin{proof}
%     Let $(z_1, \dots, z_{k+1}) \in G^{k+1}$. Then the Loomis--Whitney inequality implies
%         \begin{align*}
%              \prod_{\ell=1}^{k+1}\left|\big(\nabla p_{z_1}(x)~|~ \cdots ~|~\nabla p_{z_{\ell-1}}(x)~|~\nabla p_{z_{\ell+1}}(x)~|~ \cdots ~|~ \nabla p_{z_k}(x)\big)\right|^{1/k}\\
%              \geq \left|\big(\nabla p_{z_1}(x)~|~\nabla p_{z_2}(x)~|~ \cdots ~|~ \nabla p_{z_{k+1}}(x)\big)\right|
%         \end{align*}
%     Therefore, since $\left|\big(\nabla p_{z_1}(x)~|~ \cdots ~|~\nabla p_{z_{\ell-1}}(x)~|~\nabla p_{z_{\ell+1}}(x)~|~ \cdots ~|~ \nabla p_{z_k}(x)\big)\right|\leq 1$ for all $\ell$, then for each $\ell$, we have
%         \begin{align*}
%             \left|\big(\nabla p_{z_1}(x)~|~ \cdots ~|~\nabla p_{z_{\ell-1}}(x)~|~\nabla p_{z_{\ell+1}}(x)~|~ \cdots ~|~ \nabla p_{z_k}(x)\big)\right| \geq c^{k}.
%         \end{align*}
% \end{proof}

% \begin{cor}
%     Let $E \subseteq \R^d$, and let $G \subseteq \R^d$ be countable. Then there exist an integer $k \in \{1, \dots, d-1\}$ such that $E$ is $k$-transversal with respect to $G$, and $E$ is not $ k+1$-transversal with respect to $G$.
% \end{cor}

\begin{cor}\label{cor twocond}
    Let $E \subseteq \R^d$, $\alpha>0$ and let $G \subseteq \R^d$ be countable. Then there exist an integer $k \in \{1, \dots, d\}$,   and $F \subset E$ such that $\pdim F \geq \pdim E-\alpha$, $F$ is $k$-transversal,  and for each $\{x_1, \dots, x_{k+1}\} \subset G$
    \begin{align*}
        \left| \big(\nabla p_{x_1}(x)~|~\nabla p_{x_2}(x)~|~ \cdots ~|~ \nabla p_{x_{k+1}}(x)\big) \right|=0.
    \end{align*}
    for all $x \in F$.
\end{cor}

\begin{proof}
    For any $x, y \in \R^d$, $\|\nabla p_y(x)\|=1$, thus every set is $(\alpha/d, 1)$-transversal with respect to $G$. By definition, there exists a constant $c_1>0$, $z \in G$ and $F_1\subset E$ (namely $F_1=E$) such that $\pdim F_1 \geq \pdim E-\tfrac{1}{d}$ and  $ \left| \big(\nabla p_{z}(x)\big)\right|\geq c_1$ for all $x \in F_1$. 

    Now $F_1$ is either $(\alpha/d, 2)$-transversal or not. If not, then Lemma \ref{lem contra} implies that there is a set $\tilde{F}_1 \subset F_1$ such that $\pdim \tilde{F}_1 = \pdim F_1 \geq \pdim E-\alpha/d$. The corollary is then proven with $\tilde{F}_1=:F$.
    
    If $F_1$ is $(\alpha/d, 2)$-transversal then there exists $F_2 \subset F_1 \subset E$, $c_2>0$ and $z_1, z_2 \in G$ such that $\pdim F_2 \geq \pdim F_1 -\alpha/d\geq \pdim E-2\alpha/d$ and  $ \left| \big(\nabla p_{z_1}(x)~|~\nabla p_{z_2}(x)\big)\right|\geq c_2$ for all $x \in F_2$. Induction completes the argument.
\end{proof}

\begin{lem}\label{lem geomlem}
    Let $x \in E  \subset \R^d$, $0< \gamma<\delta \ll L < 1$, and $k \in \Z_+$. Let $G \subseteq E$ be a countable, dense subset of $E$.
    Suppose that for each $\{x_1, \dots, x_{k+1}\} \subset G$, 
        \begin{align}\label{eq degen}
           \left| \big(\nabla p_{x_1}(x)~|~\nabla p_{x_2}(x)~|~ \cdots ~|~ \nabla p_{x_{k+1}}(x)\big) \right|=0
        \end{align}
    and there exists $\{z_1, \dots, z_k\} \subseteq G$ such that
        \begin{align}\label{eq strongtranvers}
           \left| \big(\nabla p_{z_1}(x)~|~\nabla p_{z_2}(x)~|~ \cdots ~|~ \nabla p_{z_{k}}(x)\big) \right|\geq L.
        \end{align}
    If $\mathbf{p}=\left(p_{z_i}\right)_{i=1}^k$, then there exists $C=C(L)$ and  $W_{x} \in G(d, k)$ such that 
        \begin{align*}
            \mathbf{p}^{-1}(B(\xi, \gamma)) \cap E \cap B(x, \delta) \subset  N_{C\cdot (\gamma+h(\delta))}(W_{x}^{\perp}+x) \cap X(x, W_{x}, \eta)
        \end{align*}
    where $\xi=\mathbf{p}(x)$.
\end{lem}

\begin{proof}
 Define
    \begin{align*}
        W_{\bar{z}}:= \Span \left\{ \nabla p_{z_1}(x), \nabla p_{z_2}(x), \dots , \nabla p_{z_{k}}(x) \right\}. 
    \end{align*}
By inequality \eqref{eq strongtranvers}, $W_{\bar{z}} \in G(d, k)$. By inequality \eqref{eq degen}, for any $y \in G$, $\nabla p_{y}(x) \in W_{\bar{z}}$.
Therefore, Lemma \ref{lem direction} implies that, for all $y \in G$, 
    \begin{align*}
        y-x \in X(0, W_{\bar{z}}, \eta)
    \end{align*}
and, since $G$ is dense in $E$,
    \begin{align*}
        E \subset X(x, W_{\bar{z}}, \eta ).
    \end{align*}

Let $\xi =\mathbf{p}(x)$. Now
    \begin{align*}
        \mathbf{p}^{-1}(B(\xi, \gamma)) \subset \bigcap_{i=1}^k \left\{y \in \R^d~|~ |\|y-z_i\|_*-\xi_i|<\gamma \right\}.
    \end{align*}
For each $i$, if $y_0 \in \left\{y \in \R^d~|~ |\|y-z_i\|_*-\xi_i|<\gamma \right\}$ and $\|y_0-x\|_*< \delta$, then 
    \begin{align*}
        \big|\|y_0-z_i\|_*- \|x-z_i\|_*\big|< \gamma
    \end{align*}
and by Lemma \ref{lem gradest},
    \begin{align*}
        \big|\|y_0-z_i\|_*- \|x-z_i\|_*\big| + h(\delta)\geq |\langle \nabla \|x-z_i\|_*, y_0-x\rangle|.
    \end{align*}
    Therefore,
    \begin{align*}
        \bigcap_{i=1}^k \left\{y \in \R^d~|~ |\|y-z_i\|_*-\xi_i|<\gamma \right\} \subset \bigcap_{i=1}^k \left\{y\in \R^d~|~ |\langle \nabla p_{z_i}(x), x-y\rangle|<\gamma +h(\delta) \right\}.
    \end{align*}
Now, inequality \eqref{eq strongtranvers} implies that there exists a constant $C=C(L)$ satisfying $C(L) \to \infty$ as $L \to 0$ such that
    \begin{align*}
        \bigcap_{i=1}^k \left\{y\in \R^d~|~ |\langle \nabla p_{z_i}(x), x-y\rangle|<\gamma +h(\delta) \right\} \subset N_{C\cdot (\gamma+h(\delta))}(W_{\bar{z}}^{\perp}+x)
    \end{align*}
This completes the proof with $W_x:=W_{\bar{z}}$.
\end{proof}

\begin{prop}\label{prop:c1weaktrans}
    Let $E \subset B(0, 1)\subset  \R^d$, $\alpha>0$, and $G \subset E$ be a countable, dense subset of $E$. Then there exists  $k \in \{1,\dots, d\}$, $\{z_1,\dots, z_k\} \subset G$, $F \subset E$, and $\delta_0>0$  such that $\pdim E -\alpha \leq\pdim F$ and for every $\delta \in (0, \delta_0)$ and $\xi \in \R^d$, there exist $m$ points $x_1,\dots,x_m \in \R^d$ such that $m =m(\|\cdot\|_*, d)$ and
    \begin{equation} \label{eq:sep}
        \mathbf{p}^{-1}\big( B(\xi,\delta) \big) \cap F \subseteq \bigcup_{n=1}^m B\big( x_n, \delta \big)
    \end{equation}
    where $\mathbf{p}=\left(p_{z_i}\right)_{i=1}^k$.
\end{prop}

\begin{proof}
Corollary \ref{cor twocond} implies that  there exist an integer $k \in \{1, \dots, d\}$,   and $F \subset E$ such that $\pdim F \geq \pdim E -\alpha$, $F$ is $k$-transversal,  and for each $\{x_1, \dots, x_{k+1}\} \subset G$
\begin{align*}
    \left| \big(\nabla p_{x_1}(x)~|~\nabla p_{x_2}(x)~|~ \cdots ~|~ \nabla p_{x_{k+1}}(x)\big) \right|=0
\end{align*}
for all $x \in F$. Since $F$ is $k$-transversal with respect to $G$, there exist $L \in (0,1)$, $\{z_1,\dots,z_k\} \subseteq G$, such that
\begin{align*}
    \left|\big(\nabla p_{z_1}(x)~|~\nabla p_{z_2}(x)~|~ \cdots ~|~ \nabla p_{z_k}(x)\big)\right|\geq L
\end{align*}
for all $x \in F$.
  
Let $x \in F$ and  $\xi=\mathbf{p}(x)$. Lemma \ref{lem geomlem} implies  that there exists $W_x \in G(d, k)$ and $C=C(L)$ such that 
\begin{align*}
    \mathbf{p}^{-1}(B(\xi, \gamma)) \cap F \cap B(x, \delta) \subset  N_{C\cdot (\gamma+h(\delta))}(W_{x}^{\perp}+x) \cap X(x, W_{x}, \eta).
\end{align*}
Recall, $\Lambda =\Lambda( \|\cdot\|_*, d)$ for the proof of Lemma \ref{lem direction}. Let $\delta_0>0$ be small enough so that $h(\delta) \leq \frac{\Lambda^2\sqrt{1-\eta^2}}{100 C}\delta$ for all $\delta\leq \delta_0$ and let $\gamma \leq \frac{\Lambda^2\sqrt{1-\eta^2}}{100C}\delta$.  

 If $y \in N_{C\cdot (\gamma+h(\delta))}(W_{x}^{\perp}+x)$,  then 
 \begin{align*}
     \|P_{W_x}(y-x)\| \leq \Lambda^{-1} \|P_{W_x}(y-x)\|_* \leq \Lambda^{-1}C\cdot (\gamma+h(\delta)).
 \end{align*} 
 Also, if $y \in X(x, W_{x}, \eta)$, then $\|P_{W_x^{\perp}}(y-x)\|< \eta\|y-x\|$ which implies $(1-\eta^2)^{1/2}\|x-y\| < \|P_{W_x}(y-x)\|$. Thus if $y \in N_{C\cdot (\gamma+h(\delta))}(W_{x}^{\perp}+x) \cap X(x, W_{x}, \eta)$, then 
 \begin{align*}
     \|x-y\|_* \leq \Lambda^{-1}\|x-y\| < \Lambda^{-2}C\cdot (\gamma+h(\delta)) \cdot (1-\eta^2)^{-1/2}.
 \end{align*}
 Therefore,
\begin{align*}
    N_{C\cdot (\gamma+h(\delta))}(W_{x}^{\perp}+x) \cap X(x, W_{x}, \eta) \subset B\!\left(x, C\cdot \tfrac{(\gamma+h(\delta))}{\Lambda^2\sqrt{1-\eta^2}}\right)\subset B(x, \tfrac{1}{50}\delta).
\end{align*}
This further implies that for $\gamma = \frac{\Lambda^{2}\sqrt{1-\eta^2}}{100C}\delta$
\begin{align}\label{eq stabletransverse}
    \mathbf{p}^{-1}(B(\xi, \gamma)) \cap F \cap B(x, \delta) \subset B(x, \tfrac{1}{50}\delta) \subset \bigcup_{i}B(y_i, \gamma)
\end{align}
for at most approximately $(2C/\Lambda^2\sqrt{1-\eta^2})^d$ balls of the form $B(y_i, \gamma)$.  

% Moreover, for $\gamma \leq\frac{1}{50}(\frac{\Lambda^{2}\sqrt{1-\eta^2}}{100C})\delta$
% \begin{align*}
%     \mathbf{p}^{-1}(B(\xi, \gamma)) \cap F \cap B(x, \delta) \subset \mathbf{p}^{-1}(B(\xi, \gamma)) \cap F \cap B(x, \tfrac{1}{50}\delta) \subset B(x, (\tfrac{1}{50})^2\delta).
% \end{align*}

Now let $\gamma \leq \frac{1}{100}\frac{\Lambda^2\sqrt{1-\eta^2}}{100C}\delta_0$ and cover $F$ with approximately $\delta_0^{-d}$ balls, $\{B_{\ell}\}$ of radius $\frac{1}{100}\delta_0$. If  $\mathbf{p}^{-1}(B(\xi, \gamma)) \cap F \cap B_{\ell} \neq \emptyset$, let $x_{\ell}\in \mathbf{p}^{-1}(B(\xi, \gamma)) \cap F \cap B_{\ell}$ and $\xi_{\ell}= \mathbf{p}(x_{\ell})$. Then triangle inequality implies
\begin{align*}
    \mathbf{p}^{-1}(B(\xi, \gamma)) \cap F \cap B_{\ell} \subset \mathbf{p}^{-1}(B(\xi_{\ell}, 100\gamma)) \cap F \cap B(x_{\ell}, \delta_0)
\end{align*}
Now \eqref{eq stabletransverse} implies that for each $\ell$, there is a collection, $\{B_{i}^{\ell}\}$, of approximately $\Big( 100\frac{2C}{\Lambda^2\sqrt{1-\eta^2}} \Big)^d$ balls of radius $\gamma$ such that 
\begin{align*}
    \mathbf{p}^{-1}(B(\xi_{\ell}, 100\gamma)) \cap F \cap B(x_{\ell}, \delta_0)\subset  \bigcup_{i}B_i^{\ell}
\end{align*}
which implies
\begin{align*}
     \mathbf{p}^{-1}(B(\xi, \gamma)) \cap F &= \bigcup_{\ell} \mathbf{p}^{-1}(B(\xi, \gamma)) \cap F \cap B_{\ell}\\
     &= \bigcup_{\ell}\mathbf{p}^{-1}(B(\xi_{\ell}, 100\gamma)) \cap F \cap B(x_{\ell}, \delta_0)\subset \bigcup_{\ell} \bigcup_{i}B_i^{\ell}
 \end{align*}

Thus inequality \eqref{eq:sep} is established with
\begin{equation*}
    m \sim \delta_0^{-d}\left(\frac{200C}{\Lambda^2\sqrt{1-\eta^2}}\right)^d. \qedhere
\end{equation*}    
\end{proof}

%%%%%%%%%%%%%%%%%%%%%%%%%%%%%%%%%%%%
%%% Section 6: Proof of C^1 Case %%%
%%%%%%%%%%%%%%%%%%%%%%%%%%%%%%%%%%%%

\section{Proof of Theorem \ref{thm:main} -- The \texorpdfstring{$C^1$}{C1} case} \label{s:C1-proof}

Let $E \subset \R^d$, let $\varepsilon>0$, and let $\|\cdot\|_*$ be a $C^1$ norm on $\R^d$. Let $s=\pdim E$ and without loss of generality, let $E \subset B(0, 1)$.
    
Proposition \ref{prop:c1weaktrans} implies that there exists $k \in \{1,\dots, d\},$ $F \subset E$ and $\{z_1,\dots, z_k \} \subset E$ such that $\pdim F \geq \pdim E - \eps$ and if $\mathbf{p}_i: F \to \R$ is defined by
\begin{align*}
    \mathbf{p}_i(x):= \|x-z_i\|_*
\end{align*}
for all $i \in \{1, \dots, k\}$, then $\mathbf{p}:=(\mathbf{p}_i)_{i=1}^k$ is weakly transversal. Therefore, Proposition \ref{prop:nonlinear-jarvenpaa} implies 
\begin{align*}
    \frac{1}{k} \pdim F \leq  \max_{1 \leq i \leq k} \pdim \mathbf{p}_{i}(F).
\end{align*}
Since $\mathbf{p}_{i}(F)= \Delta_{z_i}^*(F)$, 
\begin{equation*}
    \frac{1}{d} \pdim E -\eps\leq \frac{1}{k} \pdim E -\eps\leq\frac{1}{k} \pdim F \leq  \max_{1 \leq i \leq k} \pdim \Delta_{z_i}^*(F)\leq  \max_{1 \leq i \leq k} \pdim \Delta_{z_i}^*(E). \QED
\end{equation*}

%%%%%%%%%%%%%%%%%%%%%%%%%%%%%%%%
%%% Section 7: Sharp Example %%%
%%%%%%%%%%%%%%%%%%%%%%%%%%%%%%%%

\section{Proof of Theorem \ref{thm:sharp}} \label{s:sharp}

\noindent Let $\| x\|_P= \max \, \{ |x \cdot v^\ell| \}_{\ell=1}^N$ be a polynomial norm such that 
    \begin{align*}
        \{v^1, \dots, v^N \} \subset \Q^d.
    \end{align*}
Then there exists $q \in \Z^+$  such that  the following holds: for each $v^j= (v^j_1, \dots,v^j_d) \in \{v^1, \dots, v^N \}$ and for each $\ell \in \{1,\dots, d\}$, there is a  $p^j_{\ell} \in \{0, \dots, q-1\}$ such that
    \begin{align*}
        v_{\ell}^j = \frac{p^j_{\ell}}{q}.
    \end{align*} 
Consider two sequences of integers $\{m_k\}_{k=1}^{\infty}$ and $\{M_k\}_{k=1}^{\infty}$ such that $M_k-m_k\geq 2^k$ for all $k$ and 
    \begin{align*}
        \limsup_{N \to \infty} \frac{\#\left([0, N] \cap \bigcup_{k=1}^{\infty}[m_k, M_k]\right)}{N} = \frac{s}{d}.
    \end{align*}
Let $A := \bigcup_{k=1}^{\infty}[m_k, M_k]$ and define $F \subset \R$ to be the compact set
\begin{align*}
    F := \left\{ x = \sum_{m=1}^{\infty} \frac{x_m}{q^m}~:~ x_m \in \{0, 1,\dots, q-1\}\, \forall m \mbox{ and } x_m=0 \mbox{ if } m \not\in A  \right\}.
\end{align*}
Then $\pdim F = \tfrac{s}{d}$, so if we let
\begin{equation*}
    E := \underbrace{F \times F \times \cdots \times F}_{d \text{ times }} \subset \R^d,
\end{equation*}
then we have $\pdim E = s$.
% Observe that for fixed $k_0,$
%     \begin{align*}
%           \frac{\#([0, M]+[m_{k_0}, M_{k_0}])  \cap \Z)}{N}\leq \frac{(M_{k_0}-m_{k_0})+M}{N}
%     \end{align*}
Now for every $j \in \{ 1, \dots, N\}$,
    \begin{align*}
       & \left\{ (x-y)\cdot v^j~:~ x, y \in E\right\} \\
        &\hspace{.1cm}= \bigcup_{\ell=1}^d \left\{  \sum_{m=1}^{\infty} \frac{(x_m-y_m)p^j_{\ell}}{q^{m+1}} : x_m, y_m \in \{0, 1,\dots, q-1\} \ \forall \+ m \mbox{ and } x_m=0 \mbox{ if } m \not\in A\right\}\\
        &\hspace{.1cm}\subseteq \left\{\sum_{m=1}^{\infty} \frac{x_m}{q^m} : x_m \in \{ -q+1,\dots, q-1\} \ \forall \+ m \mbox{ and } x_m=0 \mbox{ if } m \not\in \bigcup_{k=1}^{\infty}[m_k, M_k+2+\log d] \right\}.
    \end{align*}
Therefore, 
\begin{align*}
    \pdim  \left\{ |(x-y)\cdot v_{\ell}^j|~:~ x \in E\right\} \leq \limsup_{N \to \infty} \frac{\#\left([0, N] \cap \bigcup_{k=1}^{\infty}[m_k, M_k+2 +\log d]\right)}{N} = \frac{s}{d}
\end{align*}
and thus
\begin{equation*}
    \pdim \Delta^*(E) \leq \frac{s}{d}. \QED
\end{equation*}

We do not know whether the hypothesis of rational dependence in Theorem \ref{thm:sharp} can be omitted. Schematically, producing an example of the sort constructed above without this dependence between the faces would require solving an overdetermined system of linear equations. As such, a hypothetical example could not rely on this ``prescribed projection" approach. Such constructions are common in the Hausdorff dimension regime and were previously used in \cite{altaf2023distance}, but prescribing projections at \textit{all} scales rather than simply a sequence of scales is fundamentally more challenging.

\vspace*{-0.03cm}

\phantomsection
\section*{Acknowledgment}

\noindent We would like to thank Krystal Taylor, Eyvindur Palsson, Alexia Yavicoli, Ben Jaye, and Malabika Pramanik for organizing and hosting the 2024 workshop \textit{On the Interface of Geometric Measure Theory and Harmonic Analysis} at the Banff International Research Station. In addition, we would like to thank all the participants for contributing to the enriching and community-building experience.

\vspace*{-0.03cm}

%%%%%%%%%%%%%%%%%%
%%% References %%%
%%%%%%%%%%%%%%%%%%

\bibliographystyle{plain}
\bibliography{references}

\end{document}